\documentclass[final,twoside,a4paper]{amsart}
\usepackage[english]{babel}
\usepackage{amsmath, amsfonts,amssymb,amsopn,amscd,amsthm,mathrsfs}
\usepackage{bm}
\usepackage{graphicx}
 \usepackage{enumerate}
 \usepackage{array}
 \usepackage{hyperref}
 \hypersetup{colorlinks=true,linkcolor=blue,citecolor=blue,linktoc=page}
\usepackage{showkeys}
\numberwithin{equation}{section}
\usepackage{fancyhdr}
\newcommand{\eq}{\begin{equation}}
\newcommand{\qe}{\end{equation}}

\newcommand{\E}{\mathbb{E}}

\newcommand{\R}{\mathbb{R}}
\newcommand{\p}{\mathbb{P}}
\theoremstyle{plain}

\newtheorem{thm}{Theorem}[section]
\newtheorem{lem}{Lemma}[section]
\newtheorem{prop}{Proposition}[section]
\newtheorem{cor}{Corollary}[section]

\theoremstyle{definition}

\theoremstyle{remark}
\newtheorem*{rem}{Remark}

\usepackage{color}

\begin{document}
\sloppy
\pagestyle{headings} 
\title{Non Asymptotic Variance Bounds and Deviation Inequalities by Optimal Transport}
\date{Note of \today}
\author{Kevin Tanguy \\ University of Angers, France}

\begin{abstract}
The purpose of this note is to show how simple Optimal Transport arguments, on the real line, can be used in Superconcentration theory. This methodology is efficient to produce sharp non-asymptotic variance bounds for various functionals (maximum, median, $l^p$ norms)  of standard Gaussian random vectors in $\R^n$. The flexibility of this approach can also provide exponential deviation inequalities reflecting preceding variance bounds. As a further illustration, usual laws from Extreme theory and Coulomb gases are studied.
\end{abstract}
\maketitle 

\section{Introduction}

As an introduction we recall some facts about Gaussian concentration of measure (cf. \cite{Led}) and Superconcentration theory (cf. \cite{Chatt1}).\\

It is well known that concentration of measure is an effective tool in various mathematical areas (cf. \cite{BLM}). In a Gaussian setting, classical concentration results typically produce, for $f\,:\,\R^n\to \R$ a Lipschitz function with Lipschitz constant $\|f\|_{{\rm Lip}}$, 

\begin{equation}\label{eq.sudakov.tsirelson}
\gamma_n\big(|f-\E_{\gamma_n}[f]|\geq t\big)\leq 2e^{-\frac{t^2}{2\|f\|_{Lip}^2}}, \quad t\geq 0,
\end{equation}

\noindent with $\gamma_n$ the standard Gaussian measure on $\R^n$. Another instance of concentration of measure is the Poincar\'e's inequality satisfied by $\gamma_n$. Namely, for $f\in L^2(\gamma_n)$ smooth enough : 

\begin{equation}\label{eq.poincare.gaussien}
{\rm Var}_{\gamma_n}(f)\leq \int_{\R^n}|\nabla f|^2d\gamma_n,
\end{equation}

\noindent where $|\cdot|$ stands for the Euclidean norm on $\R^n$. As effective as \eqref{eq.sudakov.tsirelson} and \eqref{eq.poincare.gaussien} are, their generality can lead to sub-optimal bounds in some particular case. For instance, consider the $1$-Lipschitz function on $\R^n$ $f(x)=\max_{i=1,\ldots,n}x_i$. At the level of the variance, \eqref{eq.poincare.gaussien} gives 

$$
{\rm Var}(M_n)\leq 1, 
$$

\noindent with $M_n=\max_{i=1,\ldots,n}X_i$ where $(X_1,\ldots,X_n)$ stands for a standard Gaussian random vector in $\R^n$, whereas it has been proven that ${\rm Var}(M_n)\leq C/\log n$ with $C>0$ a numerical constant. At an exponential level, \eqref{eq.sudakov.tsirelson} is not satisfying either. Indeed, it is well known in Extreme theory (cf. \cite{Lead}) that $M_n$ can renormalized by some numerical constants,   $a_n=\sqrt{2\log n}$ and $b_n=a_n-\frac{\log 4\pi+\log\log n}{2a_n}$, $n\geq 1$, such that

$$
a_n(M_n-b_n)\to \Lambda_0
$$

\noindent in distribution, as $n\to\infty$, where $\Lambda_0$  corresponds to the Gumbel distribution :

$$
\p(\Lambda_0\leq x)=\exp(-e^{-x}), \quad x\in \R.
$$

\noindent Then, it is clear that the asymptotics of $\Lambda_0$ are not Gaussian but rather exponential on the right tail and double exponential on the left tail. It is now obvious that \eqref{eq.sudakov.tsirelson} and \eqref{eq.poincare.gaussien} lead to sub-optimal results for the function $f(x)=\max_{i=1,\ldots,n}x_i$. When such phenomenon happens it is referred as Superconcentration phenomenon (cf. \cite{Chatt1}). This kind of phenomenon could be seen for different functionals of Gaussian random variables (and also, as we will see, for other laws of probability) and as been studied in \cite{BT, KT, KT2, Paou, Val}\ldots.\\

The purpose of this note is to show how simple transport arguments on the real line can easily lead to weighted Poincar\'e's inequalities together with deviation inequalities which are relevant in Superconcentration theory. In particular, we will emphasize the fact that such results can be obtained by transporting the Exponential measure toward the measure of interest.\\

Let us describe the setting of our work before stating our main results. Let $\mu$ and $\nu$ be two probability measures on $\R$. Assume that both of these measures are absolutely continuous with respect to the Lebesgue measure on $\R$. More precisely, assume that there exists two smooth functions $g\,:\,\R\to\R$ and $h\,:\,\R\to\R$ such that
 
$$
d\mu(x)=h(x)dx,\quad d\nu(x)=g(x)dx
$$

Then, let $X$ be a random variable with law $\mu$ and $Y$ be a random variable with law $\nu$. Denote by $H$ (respectively by $G$) the cumulative distribution function of $X$ (respectively $Y$) and define the hazard function associated to the probability measure $\mu$ by

$$
\kappa_\mu(x)=\frac{h(x)}{1-H(x)}, \quad x\in {\rm supp}(\mu)\subset \R.
$$

\noindent Similarly, $\kappa_\nu$ will be the hazard function associated to $\nu$.\\

\noindent Besides, we will also assume that $\nu$ satisfies a Poincar\'e inequality on $\R$ with constant $C_\nu>0$. That is to say, for $f\,:\,\R\to\R$ smooth enough, 

$$
{\rm Var}_{\nu}(f)\leq C_\nu \int_{\R}f'^2d\nu.
$$
\begin{rem}
It is classical (cf. \cite{Led}) that $\nu^n=\nu\otimes\ldots\otimes\nu$ will also satisfy a Poincar\'e's inequality with the same constant $C_\nu$. 
\end{rem}

We will denote by $T\,:\,\R^n\to\R^n$ the transport map between $\mu^n$ and $\nu^n$. It satisfies, for any Borelian function $f\,:\,\R^n\to\R$,

$$
\E_{\mu^n}(f)=\E_{\nu^n}\big(f\circ T\big).
$$

\noindent and $T(x_1,\ldots,x_n)=\big(t(x_1),\ldots,t(x_n)\big)$ with $t\,:\, \R\to\R$ the monotone rearrangement map pushing $\nu$ toward $\mu$ (cf. section two).\\

In the sequel of this note (unless otherwise stated), $Y=(Y_1,\ldots, Y_n)$ will stand for a random vector in $\R^n$ with $\mathcal{L}(Y)=\nu^n$ and $X=(X_1,\ldots,X_n)$ for a random vector in $\R^n$  with  $\mathcal{L}(X)=\mu^n$.\\

Now, let us state our main results.

\begin{thm}\label{thm.tanguy.transport1}
With the preceding notations, for any function $f\,:\,\R^n\to\R$ smooth enough, $n\geq 1$, we have

\begin{equation}
{\rm Var}\big(f(X)\big)\leq C_\nu\sum_{i=1}^n\E\bigg[(\partial_if)^2\circ T(Y)\bigg(\frac{\kappa_\nu(Y_i)}{\kappa_\mu\big(t(Y_i)\big)}\bigg)^2\bigg],
\end{equation}
\end{thm}

As we will see, preceding Theorem can be used to obtain exponential deviation inequality for $M_n=\max_{i=1,\ldots,n}X_i$.

\begin{thm}\label{thm.inegalite.deviation.transport}

Assume that there exists a function  $x\mapsto \psi(x)$ from $\R$ to $\R$, non-increasing such that

$$
\bigg|\frac{\kappa_\nu\big(t^{-1}(x)\big)}{\kappa_\mu(x)}\bigg|\leq \psi(x),\quad x\in \R
$$

\noindent and there exists $\epsilon_n$ such that

$$
\E\bigg[\psi(M_n)^2\bigg]\leq \epsilon_n.
$$

\noindent Then, for any $t\geq 0$ and $n\geq 1$,

$$
\p(\sqrt{\epsilon_n}\big(M_n-\E[M_n]\big)\geq t)\leq 3e^{-t} .
$$
\end{thm}
\begin{rem}
As it will be clear in the sequel, the arguments can also be performed for any other order statistics obtained from the random vector $X$.
\end{rem}
To ease the understanding of our results, we give below an application of them when $\nu$ is the (symmetric) Exponential measure on $\R$ and $\mu$ is the standard Gaussian measure $\gamma_1$ on $\R$.

\begin{prop}\label{prop.poincare.poids.gaussien}
For $f\,:\,\R^n\to\R$ smooth enough and $n\geq 1$, we have

$$
{\rm Var}_{\gamma_n}(f)\leq C\sum_{i=1}^n\E_{\gamma_n}\bigg[(\partial_i f)^2(X)\bigg(\frac{1}{1+|X_i|}\bigg)^2\bigg]
$$

\noindent with $C>0$ a numerical constant.
\end{prop}

\noindent In particular, applied to (a smooth approximation of) $f(x)=\max_{i=1,\ldots,n}x_i$, we get, for every $n\geq 1$,

\begin{equation}\label{eq.poincare.poids.gaussien}
{\rm Var}(M_n)\leq C\E\bigg[\frac{1}{1+M_n^2}\bigg]\leq \frac{C}{1+\log n}
\end{equation}

\begin{prop}
The following deviation inequality holds, for any $n\geq 1$,
\begin{equation}\label{eq.gaussien.deviation.droite}
\gamma_n\big(M_n-\E[M_n]\geq t)\leq 3e^{-ct\sqrt{\log n}},\quad t\geq 0
\end{equation}
\end{prop}
\begin{rem}
Notice that preceding results improve upon classical concentration of measure (namely \eqref{eq.sudakov.tsirelson} and \eqref{eq.poincare.gaussien}) and can also be used for other functionals such as the Median.
\end{rem}

Throughout all the article $C$ will stand for a positive numerical constant which may change at each occurence. 
\section{Tools and proofs of the main results}
\subsection{Basics facts}
First, let us expose the elementary tools from Optimal Transport, on the real line, that will be needed in the sequel. 
\noindent  We want to highlight the fact that we will mostly choose (in practice) $\nu$ as the Exponential measure on $\R_+$ (or as the symmetric Exponential mesure on $\R$) from which we will improve some concentration properties satisfied by the measure of interest $\mu$. However, when stated, we will not specify the measure $\mu$ and $\nu$ in our results.  \\
\newline
Recall that the monotone transport from $\nu$ to $\mu$ (cf. \cite{Vil2} for more details) is obtained by an application  $t:\R\to\R$ such that, for every $x\in \R$,

\begin{equation}\label{eq.monge.ampere1}
G(x)=\p(Y\leq x)=\int_{-\infty}^xd\nu=\int_{-\infty}^{t(x)}d\mu=\p\big(X\leq t(x)\big)=H\big(t(x)\big),\quad x\in\R.
\end{equation}

\noindent Which leads, after differentiation, to the following equality

\begin{equation}\label{eq.monge.ampere2}
g(x)=h\big(t(x)\big)t'(x),\quad x\in \R
\end{equation}

\noindent Then, the application $T\,:\,\R^n\to\R^n$ defined by $T(x)=\big(t(x_1),\ldots,t(x_n)\big)$, for every $x=(x_1,\ldots,x_n)\in~\R^n$  transports $\nu^n$ on $\mu^n$. In particular, for any $f\,:\,\R^n\to \R$ smooth enough,

$$
{\rm Var}_{\mu^n}(f)={\rm Var}_{\nu^n}\big(f\circ T\big).
$$

The following Lemma (cf. \cite{Led}) will also be useful in the sequel.

\begin{lem}\label{lem.poincare.exponentiel.transport}

Let $X$ a centered random variable such that, for any  $0<\theta<\frac{1}{2\sqrt{K_n}}$, 

$$
{\rm Var}(e^{\theta X/2})\leq \frac{\theta^2}{4}K_n\E[e^{\theta X}],
$$

\noindent then $\p(X\geq t\sqrt{K_n})\leq 3e^{-ct}$ for every $t\geq 0$, with $c>0$ a numerical constant.
\end{lem}

\begin{rem}
This Lemma has been fruitfully used in recent articles about Superconcentration (cf. \cite{DHS1,DHS2, KT}).
\end{rem}

\noindent The preceding Lemma will be combined with Harris's negative association inequality (cf. \cite{BLM}) in order to prove the deviation inequality from Theorem \ref{thm.inegalite.deviation.transport}. \\

Now, let us state Harris's result. Recall that a  fonction $f\,:\,\R^n\to \R$ is considered to be non-increasing (respectively non-decreasing) if it is non-increasing, (respectively non-decreasing) in each coordinates while the others are fixed.

\begin{prop}\label{prop.association.negative.harris}[Harris]
Let  $f\,:\,\R^n\to\R$ be a non-decreasing function and $g\,:\,\R^n\to\R$ be a non-increasing function . Let $X_1,\ldots,X_n$ be independant random variables and set $X=(X_1,\ldots,X_n)$. Then
 
\begin{equation}\label{eq.association.negative.harris}
\E\big[f(X)g(X)\big]\leq\E\big[f(X)\big]\E\big[g(X)\big].
\end{equation}
\end{prop} 

\begin{rem}
 As we will explain in details later, Harris' negative association was a crucial argument in \cite{BT} when they studied order statistics.
\end{rem}

\subsection{Variance bounds}
We give below the proof of Theorem \ref{thm.tanguy.transport1}.

\begin{proof}
Since $T$ transports $\nu^n$ on $\mu^n$, we have

$$
{\rm Var}_{\mu^n}(f)={\rm Var}_{\nu^n}\big(f\circ T\big).
$$

\noindent Then, one can apply the Poincar\'e's inequality, satisfied by the measure $\nu^n$, to the function $f\circ T$  :

$$
{\rm Var}_{\nu^n}(f\circ T)\leq C_P \sum_{i=1}^n\int_{\R^n}(\partial_i f)^2\circ T(x)t'^2(x_i)d\nu^n(x).
$$

\noindent Besides, relation \eqref{eq.monge.ampere2} yields that

$$
t'(x)=\frac{g(x)}{1-G(x)}\times \frac{1-H\big(t(x)\big)}{h\big(t(x)\big)}=\frac{\kappa_\nu(x)}{\kappa_\mu\big(t(x)\big)}, \quad x\in \R
$$

\noindent under the condition that $h(x)>0,x\in\R$.

\end{proof}

\begin{rem}
As we will see on the examples, the important step will be to estimate the behaviour of the transport map $t$ in order to get some relevant bound on the variance of various functionals.\\

 Notice that this approach is reminiscent of some previous work of Barthe and Roberto \cite{Bar} or Gozlan \cite{Goz}  on the so-called weighted Poincar\'e's inequalities on the real line. Although our approach is similar in nature, the method of Barthe and Roberto relies on Hardy's inequality whereas ours is based on monotone rearrangement argument on the real line. Our methodology is very similar to Gozlan's work \cite{Goz} (in his article the transport map $T$ is denoted by $\omega^{-1}$).
 \end{rem}

\subsection{Deviation inequality}
Now, let us prove Theorem \ref{thm.inegalite.deviation.transport} with the combination of Theorem \ref{thm.tanguy.transport1} together with Lemma \ref{lem.poincare.exponentiel.transport} and Proposition \ref{prop.association.negative.harris}.\\

 Recall that, given an i.i.d. sample $X_1,\ldots,X_n$ with common law $\mu$ we define $M_n,\,n\geq 1,$ as 
$$
M_n=\max_{i=1,\ldots,n}X_i.
$$

\begin{proof}[Theorem \ref{thm.inegalite.deviation.transport}]
For any $\theta>0$, apply Theorem \ref{thm.tanguy.transport1} to, a suitable approximation of, the function $e^{\theta f}$ with $f(x)=\max_{i=1,\ldots,n}x_i$ and notice that the partial derivatives $\partial_if=1_{A_i}$ with $A_i=\{x_i=\max_{j=1,\ldots,n}x_j\}$, for $ i=1,\ldots,n$, form a partition of $\R^n$ (that is to say $\sum_{i=1}^n1_{A_i}=1$). So, it yields 

\begin{eqnarray*}
{\rm Var}(e^{\theta M_n/2})&\leq& C\frac{\theta^2}{4}\sum_{i=1}^n\E\bigg[1_{A_i}\psi(X_i)^2e^{\theta M_n}\bigg]\\
&=&C\frac{\theta^2}{4}\E\bigg[e^{\theta M_n}\psi(M_n)^2\bigg],
\end{eqnarray*}

\noindent with $M_n=\max_{i=1,\ldots,n} X_i$. Then, under the hypothesis of Theorem \ref{thm.inegalite.deviation.transport}, use Harris's inequality \eqref{prop.association.negative.harris}. Thus, 

\begin{eqnarray*}
{\rm Var}(e^{\theta M_n/2})&\leq&C\frac{\theta^2}{4}\E[e^{\theta M_n}]\E\big[\psi(M_n)^2\big]\\
&\leq& C\frac{\theta^2}{4}\epsilon_n\E[e^{\theta M_n}]\\
\end{eqnarray*}





The conclusion follows easily with Lemma \ref{lem.poincare.exponentiel.transport}.
\end{proof}

\section{Applications}

In this section, we provide some applications, in different mathematical areas, of Theorem \ref{thm.tanguy.transport1} and Theorem \ref{thm.inegalite.deviation.transport}.

\subsection{Extreme Theory}

We refer to \cite{Lead,DeHaan} for more details about Extreme Theory. Recall that, given a probability  measure $\mu$ (absolutely continuous with respect to the Lebesgue measure) and an i.i.d. sample $X_1,\ldots X_n$ with $\mathcal{L}(X_1)=\mu$, it is a classical fact that one can find renormalizing constants $a_n$ and $b_n$ such that $a_n(M_n-b_n)$  (where $M_n=\max_{i=1,\ldots,n}X_i$) converges in distribution as $n\to\infty$ and the limiting distributions are now fully caracterized. We will show that our main results are revelant to produce non-asymptotic variance bounds and deviation inequality in accordance to Extreme Theory.\\

Let us begin at the level of the variance.

\subsubsection{Non-asymptotics variance bounds}
 
Let us start with a pedagogical example from the Weibull's domain of attraction. To do so, we choose $\nu$ as the standard Exponential measure on $\R_+$ (that is to say $H(x)=1-e^{-x}$ if $x\geq 0$, $H(x)=0$ otherwise). Then, Theorem \ref{thm.tanguy.transport1} yields the following Corollary

\begin{cor}\label{exp}
If $Y$ follows a standard Exponential distribution on $\R_+$ then, for any function $f\,:\,\R^n\to\R$ smooth enough and every $n\geq 1$,

\begin{equation}
{\rm Var}\big(f(X)\big)\leq 4\sum_{i=1}^n\E\bigg[\bigg(\frac{\partial_if(X)}{\kappa_\mu(X_i)}\bigg)^2\bigg],
\end{equation}

\noindent where $X_1,\ldots,X_n$ are independant random variables with distribution $\mu$.\\
\newline

\noindent In particular, for (any smooth approximation of) $f(x)=\max_{i=1,\ldots,n} x_i$, 

\begin{equation}\label{PSC}
{\rm Var}\big(M_n\big)\leq C \E\bigg[\bigg(\frac{1}{\kappa_\mu(M_n)}\bigg)^2\bigg], 
\end{equation}
where $M_n=\max_{i=1,\ldots, n} X_i$ and $C>0$ is a numerical constant.
\end{cor}

\noindent In particular, if $\mu$ stands for the uniform measure on $[0,1]$ we have 

$$
{\rm Var}(M_n)\leq 4\E[(1-M_n)^2]=O(1/n^2).
$$

\begin{proof}
The first part is a straightforward application of Theorem \eqref{thm.tanguy.transport1}.\\

Now, If $\mu$ stands for the uniform measure on $[0,1]$ we have  $\kappa_{\mu}(x)=1_{x\in[0,1]}\frac{1}{1-x}$. Therefore,

\[
{\rm Var}(M_n)\leq 4\E[(1-M_n)^2]
\]

It is now an easy task to show that the preceding inequality is sharp. Indeed, for any $t\in[0,1]$, $\p(M_n\leq t)=t^n$ this implies that the maximum $M_n$ admits $t\mapsto nt^{n-1}1_{[0,1]}$ as density with respect to the Lebesgue measure. \\
\newline
\noindent Thus, 

$$
\E[M_n]=\int_0^1nt^ndt=\frac{n}{n+1}
$$

\noindent and 

$$
\E[M_n^2]=\int_0^1nt^{n+1}dt=\frac{n}{n+2}.
$$

\noindent Therefore, ${\rm Var}(M_n)=\frac{n}{n+1}-\frac{n^2}{(n+1)^2}=\frac{2n}{(n+2)(n+1)^2}=O(1/n^2)$. The same estimates also imply that 

$$
\E[(1-M_n)^2]=O\big(1/n^2\big)
$$
 \end{proof}
\begin{rem}
\begin{enumerate}
\item Recall that,  $n(M_n-1)$ converge in law toward the Weibull distribution. So, the preceding bound is the correct order of the variance of $M_n$.\\

\item More generally, if $\mu$ stands for the Beta law with parameter $a,b>0$, it is not difficult to show that, for every $x\in[0,1]$,

\begin{eqnarray*}
 \frac{1}{\kappa_{\mu}(x)}&=&\frac{\int_x^1(1-t)^{b-1}t^{a-1}dt}{x^{a-1}(1-x)^{b-1}}\\
 &\leq& \min\bigg(\frac{1}{a}\frac{(1-x^a)}{x^{a-1}},\frac{1}{b}\frac{(1-x)}{x^{a-1}}\bigg)
 \end{eqnarray*}
 
 \noindent Notice that if  $a=b=1$ we recover the estimates for the uniform measure. When $a>0$ and $b>0$  it seems hard to achieve the expected bound (of order $n^{-b}$) on the variance from the preceding estimate of $\kappa_\mu$. \\
 
\item It is also possible to send the standard exponential measure toward the Par\'eto distribution (which belongs to the Fr\'echet domain of attraction), however this leads to a trivial bound which is not really relevant.\\
\end{enumerate}
\end{rem}

Now, let us focus on the domain of attraction of the Gumbel distribution. To this task, we will transport the symmetric Exponential mesure (on $\R$) $\nu$ towards strictly log-concaves measure $\mu$ (on $\R$) (the standard Gaussian measure for instance).\\

Recall that $\nu$ admits the following density $g(x)=\frac{1}{2}e^{-|x|}$ with respect to the Lebesgue and admits $G(x)=\frac{1}{2}e^x$ if $x\leq 0$, $G(x)=1-\frac{1}{2}e^{-x}$ if $x>0$ as a cumulative distribution function. Elementary calculus yields that

\begin{equation}\label{eq.fonction.risque.exponentielle.symetrique}
\kappa_\nu(x)=
\left\{
\begin{array}{ll}
1,\quad x> 0,\\
\frac{1}{2e^{-x}-1},\quad x\leq 0.
\end{array}
\right.
\end{equation}

\noindent  Thus, Theorem \ref{thm.tanguy.transport1}  implies the following Corollary
\begin{cor}\label{2exp}
If $Y$ follows the symmetric Exponential distribution on $\R$ then, for any functions $f\,:\,\R^n\to\R$ smooth enough,
 
\begin{equation}
{\rm Var}\big(f(X)\big)\leq 4\sum_{i=1}^n\E\bigg[\partial_i^2f(X)\bigg(\frac{\kappa_{\nu}\big(t^{-1}(X_i)\big)}{\kappa_\mu(X_i)}\bigg)^2\bigg].
\end{equation}
where $X=(X_1,\ldots,X_n)$ has for distribution $\mu^n$.
\end{cor}

\begin{rem}
 Here, the constant $4$ stands for the Poincar\'e constant of the symmetric Exponential measure (cf. \cite{BGL}).
\end{rem}

To illustrate the preceding Corollary, we will need a technical Lemma. This one is a precise estimation of the behaviour of the transport function which will permit to obtain relevant bounds for the variance of the maximum of symmetric (strictly) log-concave measure  $d\mu(x)=e^{-V(x)}Z^{-1}dx$ with $Z$ a normalizing constant (\textit{e.g.} $V(x)=|x|^{\alpha}/\alpha,\, \alpha>1$).

\begin{lem}\label{Log}
Consider the transport map $t$ sending the symmetric of the Exponential measure $\nu$ toward the measure $d\mu(x)=e^{-V(x)}Z^{-1}dx$, where $V(x)=|x|^{\alpha}/\alpha,\, \alpha>1$. Then, the following holds 
$$
|t'\circ t^{-1}(x)|\leq \frac{C_\alpha}{V'(|x|)+1},\quad x\in\R
$$

\noindent with $C_\alpha>0$ a numerical constant only depending on $\alpha$.
\end{lem}

\begin{proof}
We would like to bound, for any $x\in\R$, the following ratio 

\begin{equation}\label{eq.log.concave.transport}
t'\circ t^{-1}(x)=\frac{\kappa_{\nu}\big(t^{-1}(x)\big)}{\kappa_{\mu}(x)},
\end{equation}

\noindent with $\kappa_\nu$ defined by \eqref{eq.fonction.risque.exponentielle.symetrique} and $\kappa_\mu(x)=e^{-V(x)}Z^{-1}\int_x^\infty e^{-V(t)}dt,\, x\in\R$. Recall that
$$
t^{-1}(x)=G^{-1}\circ H(x), \,x\in\R
$$

\noindent  with

\begin{equation}
G^{-1}(y)=
\left\{
\begin{array}{ll}
\ln(2y), \quad 0\leq y\leq 1/2,\\
\ln\bigg(\frac{1}{2(1-y)}\bigg),\quad 1/2\leq y\leq1.
\end{array}
\right.
\end{equation}

Let $A>0$ be sufficiently large. For $x>A$, the equation \eqref{eq.log.concave.transport} is easily bounded by standard estimates (cf. \cite{Ane}), we get

$$
|t'\circ t^{-1}(x)|=e^{V(x)}\int_x^\infty e^{-V(t)}dt\leq \frac{C}{V'(x)},
$$

\noindent avec $C>0$.\\
\newline
For $x$ belonging to the compact $[0,A]$, there exists $C>0$ such that $|t'\circ t^{-1}(x)|\leq C$. To sum up, 

$$
|t'\circ t^{-1}(x)|\leq \frac{C}{V'(x)+1},\,x>0
$$

\noindent For $x=0$ we have $|t'\circ t^{-1}(x)|=1$ since  $t^{-1}(0)=G^{-1}\circ H(0)=G^{-1}(1/2)=0$ by symmetry.\\
\newline
Now if, $x<-A$, we get 

$$
|t'\circ t^{-1}(x)|\leq \frac{2e^{V(x)}}{2e^{-t^{-1}(x)}-1}
$$

\noindent since $\frac{1}{\kappa_{\mu}(x)}=\frac{e^{V(x)}}{\int_x^\infty e^{-V(t)}dt}\leq 2e^{V(x)}$ for $x\geq 0$.\\
\newline
\noindent So, it is enough to bound from above $t^{-1}(x)$ when $x<-A$ in order to conclude. Using the symmetry of the law $\mu$, we obtain
\begin{eqnarray*}
t^{-1}(x)\leq \ln \big(2H(x)\big)&=&\ln \bigg( 2 \big[1-H(-x)\big]\bigg)\\ 
&\leq &\ln \bigg[ \frac{2e^{-V(-x)}}{V'(-x)}\bigg ].
\end{eqnarray*}

\noindent Thus, for $x<-A$, 

$$
|t'\circ t^{-1}(x)|\leq \frac{2e^{V(-x)}}{V'(-x)e^{V(-x)}-1}\leq \frac{C}{V'(-x)}.
$$

\noindent Similarly, when $-A\leq x\leq 0$, we also obtain that $|t'\circ t^{-1}(x)|\leq C$

\noindent Finally, all of this can be rewritten as follows

$$
\bigg|\frac{\kappa_\nu(t^{-1}(x))}{\kappa_\nu(x)}\bigg|\leq \frac{C}{V'(|x|)+1},
$$

\noindent with $C>0$.
\end{proof}







If $V$ is the quadratic potential associated to the standard Gaussian measure, we obtain, thanks to Lemma \ref{Log} and Corollary \ref{2exp}, the following  result (as announced in the introduction).

\begin{prop}\label{prop.poincare.poids.gaussien}
For $f\,:\,\R^n\to\R$ smooth enough, we have

\begin{equation}\label{eq.poincare.poids.gaussien}
{\rm Var}_{\gamma_n}(f)\leq C\sum_{i=1}^n\E_{\gamma_n}\bigg[(\partial_i f)^2(X)\bigg(\frac{1}{1+|X_i|}\bigg)^2\bigg]
\end{equation}

\noindent In particular, applied to (a smooth approximation of) $f(x)=\max_{i=1,\ldots,n}x_i$, we get, for every $n\geq 1$,

\begin{equation}\label{eq.poincare.poids.gaussien.maximum}
{\rm Var}(M_n)\leq C\E\bigg[\frac{1}{1+M_n^2}\bigg]\leq \frac{C}{1+\log n}
\end{equation}
\end{prop}

\begin{rem}
Notice that inequality \eqref{eq.poincare.poids.gaussien} has been already obtained, in dimension one, in  \cite{BobLed, BobHou2}.
\end{rem}

\begin{proof}
Indeed, for the function maximum, $\partial_i f=1_{A_i},\,i=1,\ldots,n$ with $A_i=~\{X_i=\max_{j=1,\ldots, n} X_j\}$ and, again, observe that $(A_i)_{i=1,\ldots,n}$ is a partition of $\R^n$. Therefore,

\begin{eqnarray*}
\sum_{i=1}^n\E\bigg[(\partial_i f)^2(X)\bigg(\frac{1}{1+|X_i|}\bigg)^2\bigg]&\leq& \E\bigg[\frac{1}{1+M_n^2}\bigg]\\
&\leq& \frac{1}{1+\log n}+\p(M_n\leq \sqrt{\log n})\\
&\leq &\frac{1}{1+\log n}+\bigg(1-\frac{\sqrt{\log n}}{1+\log n}e^{-\log n/2}\bigg)^n\\
&\leq& \frac{C}{1+\log n}
\end{eqnarray*}

\noindent Since, for every $t\geq 0$, $\p(M_n\leq t)=\big(1-\p(X_1>t)\big)^n$ with $X_1$ a Gaussian standard random variable. Then,  we can use the following estimate  (cf. \cite{Paou} (Lemma 2.5) or the appendix in \cite{Chatt1}) to bound the preceding quantity :  for any $t\geq0$, 
$$
\p(X_1>t)\geq \frac{t}{\sqrt{2\pi}(1+t^2)}e^{-t^2/2}.
$$

\noindent Thus, ${\rm Var}(M_n)\leq \frac{C}{\log n}$.
\end{proof}

\begin{rem}
Let us make few remarks on what preceed. 

\begin{enumerate}
\item As mentionned in the introduction, $\sqrt{2\log n}(M_n-b_n)$ converge, when $n\to\infty$, in law toward the Gumbel distribution (the precise value of $b_n$ is irrelevant here but can be found in \cite{DeHaan, Lead}). So, the preceding Corollary gives a non-asymptotic variance bound of the maximum in accordance with Extreme theory. Besides, such a bound   is classically obtained by hypercontractive and interpolation arguments  (cf. \cite{Chatt1}). Here, we provide an alternative proof based on Optimal Transport arguments.\\

\item Let us further notice that the scheme of proof can also be performed for the function $f(x)={\rm Med}(x_1,\ldots,x_n)$, $n\geq 1$,

\begin{eqnarray*}
{\rm Var}\big({\rm Med}(X)\big)&\leq& \frac{C}{1+n}+C\p({\rm Med}(X)\leq\sqrt n)^{n/2}\\
&\leq & \frac{C}{1+n}+o\bigg(\frac{1}{1+n}\bigg)\leq \frac{C}{1+n}
\end{eqnarray*}

\noindent which correspond to the correct order of magnitude of the variance of the median (cf. \cite{BT}). Notice that, as far as we know, such bounds can not be obtained by hypercontractive arguments.\\
\end{enumerate}
\end{rem}

More generally, if  $V(x)=|x|^\alpha/\alpha,\quad\alpha>1$.  The same proof, together with the Lemma \ref{Log}, yields

\begin{cor}
$$
{\rm Var}(M_n)\leq C\sum_{i=1}^n\E\bigg[(\partial_i f)^2(X)\bigg(\frac{1}{1+V'\big(|X_i|\big)}\bigg)^2\bigg]
$$

\noindent In particular, apply to (a smooth approximation of) $f(x)=\max_{i=1,\ldots,n}x_i$, it gives, for $n\geq N_0$ sufficiently large,

\begin{equation}\label{eq.poincare.poid.logconcave}
{\rm Var}(M_n)\leq C\E\bigg[\frac{1}{V'^2(M_n)+1}\bigg]\leq \frac{C}{1+C_\alpha \ln(n)^{2(\alpha-1)/\alpha}},
\end{equation}

\noindent with $C_\alpha>0$ and $C>0$ some numerical constants.
\end{cor}

\begin{proof}
\begin{eqnarray*}
\E\bigg[\frac{1}{1+|M_n|^{2(\alpha-1)}}\bigg]&\leq& \frac{1}{1+(\log n)^{2(\alpha-1)/\alpha}}+\p(M_n\leq (\ln n)^{1/\alpha}) \\
&\leq&\frac{1}{1+(\log n)^{2(\alpha-1)/\alpha}}+[1-\p(X_1\geq (\ln n)^{1/\alpha})]^n\\
&\leq& \frac{1}{1+(\log n)^{2(\alpha-1)/\alpha}}+\bigg(1-\frac{1}{2(\log n)^{(\alpha-1)/\alpha}n^{1/\alpha}}\bigg)^n\\
&\leq& \frac{C}{1+C_\alpha(\log n)^{2(\alpha-1)/\alpha}}
\end{eqnarray*}

Since, if $X_1$ stands for a random variables with law $\mu$, we can proceed as the Gaussian case. Indeed, $\p(X_1\geq t)\sim \frac{1}{t^\alpha-1}e^{-t^\alpha/\alpha}$ as $t\to\infty$. In particular, for $t$ large enough, this yields that $\p(X_1\geq t)\geq \frac{1}{2t^{\alpha-1}}e^{-t^\alpha/\alpha}$. 
\end{proof}

\begin{rem}

\noindent Following the proof (when $\alpha=2$ ) in \cite{Lead}, it can be easily proved that 

$$
a_n(M_n-b_n)\to \Lambda_0,
$$

\noindent in law, when $n\to\infty$, with $
a_n=\sqrt{\alpha(\log n)^{2(\alpha-1)/\alpha}}$ et $b_n=(\log n)^{1/\alpha}-\frac{\log(\alpha Z)+\frac{\alpha-1}{\alpha}\log \log n}{(\log n)^{(\alpha-1)/\alpha} }$.\\
\newline
Therefore, Corollary gives a non-asymptotic bound of the variance of the maximum reflecting this convergence result. We want to highlight the fact that such bound is another example of the Superconcentration phenomenon. Nevertheless, as far as we know, such estimates can not be obtained by hypercontractive methods (when $\alpha>2$) as the Gaussian case.
\end{rem}

\subsubsection{Deviation inequalities}

It is possible to use the preceding variance bounds to immediately obtain deviation inequalities thanks to Theorem \ref{thm.inegalite.deviation.transport}.

\begin{prop}\label{prop.deviation.gaussian}
The following deviation inequality holds, for any $n\geq 1$,
\begin{equation}\label{eq.gaussien.deviation.droite}
\gamma_n\big(M_n-\E[M_n]\geq t)\leq 3e^{-ct\sqrt{\log n}},\quad t\geq 0
\end{equation}
\end{prop}

\begin{rem}
\begin{enumerate}
\item Concerning Extreme theory, notice that this Theorem is only relevant if  $\mu$ belongs to the domain of attraction of the Gumbel distribution. Indeed,  the right tail of the Gumbel distribution behaves like $t\mapsto e^{-t}$ (whereas the left tail goes faster to $0$ with the following asymptotic : $t\mapsto e^{-e^t}$).\\
\item Proposition \ref{prop.deviation.gaussian} still holds if one substitute $\gamma_n$ with $d\mu(x)=e^{-V(x)}Z^{-1}dx$, where $V(x)=|x|^{\alpha}/\alpha,\, \alpha>1$ and the bound from \eqref{eq.poincare.poid.logconcave}  instead.\\
\item Similar results can be also be obtained if one replace the maximum by another order statistics.\\
\end{enumerate}
\end{rem}

\subsection{Variance of $l^p,p\geq 2$ norm of standard Gaussian vector}

As a further illustration of our approach, we propose to recover some variance's bounds of $l^p$-norms, $p\geq 1$, of a standard Gaussian vector, obtained in \cite{Paou}. The proof will be based on Proposition \ref{prop.poincare.poids.gaussien}. We will adopt the following notations : given a vector  $x=(x_1,\ldots,x_n)\in\R^n$ we denote by $\|x\|^p_p=\sum_{i=1}^n|x_i|^p$ its norm. \\

 In the article of  Paouris \textit{et al.} \cite{Paou}, the authors have noticed that the variance of $\|X\|_p$ is not precisely estimated by classical concentration theory. More precisely, classical tools from the theory of concentration of measure such as Poincar\'e's inequality or the isoperimetric Gaussian inequality yields the following bound

$$
{\rm Var}(\|X\|_p)\leq \max (n^{2/p-1},1),\quad p\geq 1.
$$

According to \cite{Paou}, this bound is only optimal when $1\leq p\leq 2$. The authors of \cite{Paou} improved this bound by using precise estimates of moments of Gaussian functionnals together with logarithmic Sobolev inequality (through the so-called Talagrand's inequality). More precisely, 

\begin{thm}[Paouris,Valettas, Zinn ]\label{prop.paouris}

Let $X$ be a standard Gaussian vector on $\R^n$ then
$$
{\rm Var}(\|X\|_p)\leq
\left\{
\begin{array}{ll}
C\frac{2^p}{p}n^{2/p-1},\quad 2<p\leq c\log n,\\
C/\log n,\quad p>c\log n,\\
\end{array}
\right.
$$
\noindent with $C,c>0$ some numerical constants which are independant of  $n$ and $p$.
\end{thm}

Here, we propose to recover Proposition \ref{prop.paouris}
 with Proposition \ref{prop.poincare.poids.gaussien}. We will only deal with the second assertion of the Proposition (the first part can be proved with similar arguments).

\begin{prop}\label{lp2}
For $n\geq N_0$, we have the following inequality
$$
{\rm Var}(\|X\|_p)\leq \frac{C}{\log n},\quad p> c\log n,
$$

\noindent with $C>0$ a numerical constant independant of $p$ and $n$.
\end{prop}
\begin{proof}
 Let $\delta>0$ be a parameter to be choosen later and denote by $B_{\infty}(0,\delta)~=~\{x\in\R^n,\, \|x\|_{\infty}< \delta\}$. Thus,
 
\begin{eqnarray*}
{\rm Var}(\|X\|_p)&\leq& C\sum_{i=1}^n\bigg(\int_{B_\infty(0,\delta)}\frac{|x_i|^{2(p-1)}}{1+|x_i|^2}\frac{1}{\|x\|_p^{2(p-1)}}d\gamma_n(x)+\int_{B^c_\infty(0,\delta)}\frac{|x_i|^{2(p-1)}}{1+|x_i|^2}\frac{1}{\|x\|_p^{2(p-1)}}d\gamma_n(x)\bigg)\\
&=&C\bigg(\sum_{i=1}^n\mathcal{I}_i+\mathcal{J}_i\bigg)
\end{eqnarray*}

We recall the following relations between $l^p$ and $l^q$ norms, for $p<q$, which will be freely used in the sequel,

$$
\|x\|_q\leq\|x\|_p\leq n^{1/p-1/q}\|x\|_q,\quad \forall x\in \R^n
$$

\noindent On one hand, since $p<2(p-1)$, 
\begin{eqnarray*}
\sum_{i=1}^n\mathcal{I}_i&\leq& \int_{B_\infty(0,\delta)}\frac{\|x\|_{2(p-1)}^{2(p-1)}}{\|x\|_{p}^{2(p-1)}}d\gamma_n(x)\\
&\leq& \p(X\in B_\infty(0,\delta))
\end{eqnarray*}

\noindent On the other hand, since $p<2(p-2)$,
 
\begin{eqnarray*}
\sum_{i=1}^n\mathcal{J}_i&\leq& \int_{B^c_\infty(0,\delta)}\frac{\|x\|_{2(p-2)}^{2(p-2)}}{\|x\|_{p}^{2(p-1)}}d\gamma_n(x)=\int_{B^c_\infty(0,\delta)}\bigg(\frac{\|x\|_{2(p-2)}}{\|x\|_{p}}\bigg)^{2(p-2)}\frac{1}{\|x\|_p^2}d\gamma_n(x)\\
&\leq &\int_{B^c_\infty(0,\delta)}\frac{d\gamma_n(x)}{\|x\|^2_p}\\
&\leq &\frac{1}{\delta^2}\p(X\in B^c_\infty(0,\delta))
\end{eqnarray*}

\noindent Furthermore, notice that the following upper bound is satisfied 
 
\begin{eqnarray*}
\p(X\in B^c_\infty(0,\delta))=\p(\max_{i=1,\ldots,n} |X_i|\geq \delta)&=&\p(\exists j\in\{1,\ldots,n\},\, |X_j|\geq \delta)\\
&\leq& n\p(|X_1|\geq \delta)\leq 2ne^{-\delta^2/2}.
\end{eqnarray*}

\noindent So far we have obtained, 
$$
{\rm Var}(\|X\|_p)\leq C\bigg(\bigg[1-\frac{\delta}{1+\delta^2}e^{-\delta^2/2}\bigg]^n+\frac{2ne^{-\delta^2/2}}{\delta^2}\Bigg)
$$

\noindent 


\noindent Then, choose $\delta=\sqrt{2\log n}$ (with $n$ large enough) to conclude. Indeed, we have

$$
\p\big(X\in B_\infty(0,\delta)\big)\leq (1-e^{-\delta^2/3})^n \sim e^{-n^{1/3}}
$$

\noindent together with

$$
\frac{2ne^{-\delta^2/2}}{\delta^2}= \frac{1}{\log n}.
$$

\noindent In other terms

$$
{\rm Var}(\|X\|_p)\leq C\bigg(o\big(\frac{1}{\log n}\big)+\frac{1}{\log n}\bigg)\leq \frac{C}{\log n},
$$

\noindent which is the result.
\end{proof}

\subsection{Coulomb gazes}
This section exposes another application of our main results in another mathematical area. We want to highlight that, in this section, the factors $\mu_i,\,i=1,\ldots,n$ (from the product measure $\mu_1\otimes\ldots\mu_n$) will not assumed to be identical. This difference justifies the separation of this section from the others.\\
\newline
Now, let us introduce few notions about Coulomb gazes and the results obtained by Chafa\"i and P\'ech\'e in \cite{ChaPe}. Let us consider a gas of charged particules $\{z_1,\ldots,z_n\}$ on the complex plane $\mathbb{C}$, confined individually by the external field $Q$ and experiencing a Coulomb pair repulsive interaction. This corresponds to the probability distribution $\mathbb{C}^n$ with density proportional to 

\begin{equation}\label{eq.coulomb.gas}
(z_1,\ldots,z_n)\in\mathbb{C}^n\mapsto \prod_{j=1}^ne^{-nQ(z_j)}\prod_{1\leq j<l\leq n}|z_j-z_k|^\beta 
\end{equation}

\noindent with $\beta>0$ is a fixed parameter and where $Q$ is a fixed smooth function.

We will focus on the particular case where $\beta=2$ and $Q(z)=V(|z|)$ with $V(t)=t^\alpha,\, t\geq0,\,\alpha\geq 1$. We are interested in the study of 

\begin{equation}\label{eq.order.stat.gas}
|z|_{(1)}\geq \ldots\geq |z|_{(n)}
\end{equation}

\noindent the order statistics of the moduli of the Coulomb gas. Notice that $|z|_{(1)}~=~\max_{1\leq k\leq n}|z_k|$.\\

In their article, the authors proved the following representation formula

\begin{thm}[Chafa\"i-P\'ech\'e]
For $\beta=2$ and under the preceding assumptions, we have the following equality in distribution

$$
\big(|z|_{(1)},\ldots,|z|_{(n)}\big)=\big(R_{(1)},\ldots, R_{(n)}\big)
$$

\noindent with $R_{(1)}\geq\ldots\geq R_{(n)}$ the order statistics associated to independent random variables $R_1,\ldots,R_n$ where $R_k$, for $k=1,\ldots,n$, has a density proportional to 

$$
t\mapsto t^{2k-1}e^{-nV(t)}1_{t\geq 0}.
$$

\end{thm}

\begin{rem}
More precisely, the case $\beta=2$ and $V(r)=r^2$ has been proved by Rider in \cite{Rid}. Chafa\"i and P\'eche extended Rider's results when $\beta=2$ and $V$ statisfies some convexity assumption together with some decay conditions at infinity. 
\end{rem}

In \cite{ChaPe}, based on the representation formula, the authors also proved an asymptotic results for $|z|_{(1)}$. This the content of next Theorem

\begin{thm}[Chafa\"i-P\'ech\'e]\label{thm.coulomb.gumbel}
Let $|z|_{(1)}=\max_{1\leq k\leq n}|z_k|$ be as in \eqref{eq.order.stat.gas}, with $\beta=2$. Suppose that $V(t)=t^\alpha$, for $t\geq 0$ and for some $\alpha\geq 1$. Set $c_n=\log n-2\log\log n-\log 2\pi$ and 

$$
a_n=2\big(\frac{\alpha}{2}\big)^{1/\alpha+1/2}\sqrt{nc_n} \quad b_n=\big(\frac{2}{\alpha}\big)^{1/\alpha}\bigg(1+\frac{1}{2}\sqrt{\frac{2c_n}{\alpha n}}\bigg).
$$

Then $\big(a_n(|z|_{(1)}-b_n)\big)_{n\geq1}$ converge in distribution, as $n\to\infty$, toward the standard Gumbel law.
\end{thm}

We will see that it is not difficult to get a non-asymptotic upper bound on the variance of  $|z_{(1)}|$,  together with a deviation inequality for our main results. A crucial step is the representation formula \eqref{eq.order.stat.gas} of $|z_{(1)}|$  : 

$$
|z_{(1)}|=\max_{i=1,\ldots,n}R_i\quad \text{in law}
$$

\noindent where $R_1,\ldots, R_n$ are independent random variables and $R_k$, for any $k=1,\ldots,n$, has a density proportionnal to 

$$
t\mapsto t^{2k-1}e^{-n t^{\alpha}}1_{[0,\infty)}(t),\,\alpha\geq 1.
$$

\noindent Then, it is possible to transport the standard Exponential measure on $\R^n_+$ toward the measure $\mu_1\otimes\cdots\otimes\mu_n$ with $\mu_k=\mathcal{L}(R_k)$ for any $k=1,\ldots,n$. Notice then, for every $k=1,\ldots,n$, that $\mu_k$ is log-concave on  $\R_{+}$ with potential  
\[
V_k(x)=n t^{\alpha} -(2k-1)\log t.
\]
\noindent So it is not difficult to prove (thanks to the estimates from \cite{Ane}) that

$$
\frac{1}{\kappa_{\mu_k}(x)}\leq\frac{C_\alpha}{nx^{\alpha-1}+1},\,x>0
$$
 
\noindent with $C_\alpha>0$ a numerical constant. Thus, Proposition  \ref{exp} yields

\begin{eqnarray*}
{\rm Var}(|z_{(1)}|)\leq \frac{C_\alpha}{n^2}\E\bigg[\frac{1}{|z_{(1)}|^{2(\alpha-1)}}\bigg] &\leq& \frac{C_\alpha}{n\log n}+C_\alpha\p\bigg(|z_{(1)}|\leq \bigg(\frac{\log n}{n}\bigg)^{1/2(\alpha-1)}\bigg)\\
&\leq & \frac{C_\alpha}{n\log n}+\prod_{i=1}^n\bigg[1-\p\bigg(R_i\geq \frac{\log n}{n}\bigg)^{1/2(\alpha-1)}\bigg)\bigg]\\
&\leq &\frac{C_\alpha}{n\log n}+o\bigg(\frac{1}{n\log n}\bigg)\\
&\leq& \frac{C_\alpha}{n\log n}
\end{eqnarray*}

\noindent Also, Theorem \ref{thm.inegalite.deviation.transport} immediatly gives the following deviation inequality 

$$
\p\bigg(\sqrt{n\log n}\big(|z_{(1)}|-\E[|z_{(1)}|\big)\geq t\bigg)\leq 6e^{-C_\alpha t},\,t\geq 0
$$
 
 \noindent where $C_\alpha>0$ is a numerical constant that does not depend on $n$. In other words, we have obtained a non asymptotic deviation inequality together with a variance bounds which are in accordance with Theorem \ref{thm.coulomb.gumbel}. That is to say, we have proven
 
 \begin{prop}
Let $\{z_1,\ldots,z_n\}$ be a Coulomb gazes with density proportional to

$$
(z_1,\ldots,z_n)\mapsto\prod_{j=1}^ne^{-nQ(z_j)}\prod_{1\leq j<k\leq n}|z_j-z_k|^2,
$$

\noindent with $Q=V(|z|)$ and $V(t)=t^{\alpha},\,\alpha\geq 1$. Then, for any $n>1$, the following holds

$$
{\rm Var}(|z_{(1)}|)\leq\frac{C_\alpha}{n\log n},
$$

\noindent with $C_\alpha>0$ a numerical constant, independent of $n$,
\noindent and

$$
\p\bigg(\sqrt{n\log n}\big(|z_{(1)}|-\E[|z_{(1)}|\big)\geq t\bigg))\leq 3e^{-C_\alpha t},\,t\geq 0,
$$

\noindent with $C_\alpha>0$ a numerical constant independent of $n$.

 \end{prop}

\section{Remarks and comparison with existing literature}
In this section, we will briefly explain how stronger functional inequalities can be used to reach the right asymptotic of the left tail in the Gumbel's domain of attraction. Then, we will compare our main results with the existing literature. 
\subsection{Few words on isoperimetric inequalities}
As we have already seen, the transport of the Exponential measure (toward a measure $\mu^n$) permit to improve some concentration's properties of the measure $\mu^n$. This phenomenon as already been observed by Talagrand in \cite{Tal3}. He used the isoperimetric inequality (involving a mixture of $l^1$ and $l^2$ balls) satisfied by the (symmetric) Exponential measure $\mu^n$ to improve the isoperimetric inequality satisfied by the standard Gaussian measure. More precisely, such improvement can be seen on the following concentration inequality

   \begin{equation}\label{eq.talagrand}
 \p\bigg(\bigg|\max_{i=1,\ldots,n}|X_i|-\sqrt{\log n}\bigg|\geq C\frac{t}{\sqrt{\log n}}\bigg)\leq Ce^{-ct},\quad t\geq 0
 \end{equation}

\begin{rem}
\begin{enumerate}
\item This type of inequality recently appeared in \cite{KT} for more general Gaussian measure.
\item This gives the correct asymptotic behaviour (with respect to Extreme Theory) of the right tail of the maximum. However, the asymptotic behaviour of the left tail, in \eqref{eq.talagrand}, is still sub-obtimal.
\end{enumerate}
\end{rem}

The symmetry of the (two sided) Exponential measure on $\R$, through Talagrand's isoperimetric inequality, seems to not make any distinctions between the left tail from the right and only gives a exponential decay. In \cite{Bob1}, Bobkov studied a different isoperimetric problem (with the standard Exponential measure and uniform enlargements $B_\infty$ instead). The lack of symmetry of the (standard) Exponential measure can be used to achieve the correct decay of the left tail on the maximum (in the Gumbel's domain of attraction).\\

More precisely, Bobkov proved the following Theorem. 

\begin{thm}[Bobkov]
Let $\nu^n$ stands for the (standard) Exponential measure on $\R_+$. Then, for every non empty ideal $A\subset \R_+^n$ such that $\nu^n(A)=\nu^n(B_\infty)=$ and every $r\geq 0$, the following inequality holds :

$$
\nu^n(A+rB_\infty)\geq \nu^n(B+rB_\infty),
$$

\noindent in other words,

$$
\nu^n(A+rB_\infty)\geq \bigg[e^{-r}\big[\nu^n(A)\big]^{1/n}+(1-e^{-r})\bigg]^n.
$$
\end{thm}

\begin{rem}
\begin{enumerate}
\item Recall that $A$ is an ideal of $\R_+^n$ if it satisfies the following condition {\center {if $x=(x_1,\ldots,x_n)\in A$, $y=(y_1,\ldots, y_n)\in \R_+^n$, $y_i\leq x_i$ for $i=1,\ldots,n$, then $y\in A$}}.\\
\item If $n\to\infty$ and $\nu^n(A)=p$ is constant (with respect to $n$), the right hand side of the preceding inequality decreases and converges toward a double exponential. That is to say 

$$
\nu^n(A+rB_\infty)\geq \exp(-e^{-r}\log(1/p)).
$$
\end{enumerate}
\end{rem}

As presented in \cite{Bob1}, it possible to achieve the following deviations inequalities for a measure $\mu^n$ by transporting the Exponential measure $\nu^n$.

\begin{thm}(Bobkov)
Let $X_1,\ldots, X_n$ be i.i.d. random variables with $\mathcal{L}(X_1)=\mu\in \mathcal{F}_0$ and set $M_n=\max_{i=1,\ldots,n}X_i$. Then, for every $p$, $0<p<1$, every $t\geq 0$, 

\begin{equation}\label{RT}
\p(M_n-m_p\geq t)\geq C\log (1/p)\exp(-ct),
\end{equation}

\begin{equation}\label{LT}
\p(M_n-m_p<- t)\leq C\exp\big(-e^{tc}\log(1/p)\big),
\end{equation}

\noindent with $m_p$ stands for the quantile of order $p$ of $M_n$ and $C, c>0$ are numerical constants.
\end{thm}

\begin{rem}
In \cite{Bob1}, there is some workable conditions which describe the set of measure $\mathcal{F}_0$. For instance Gamma measure or absolute value of standard Gaussian measure belong to $\mathcal{F}_0$.
\end{rem}

In particular, if we choose $p$ such that $p^{1/n}=F^{-1}(1-1/n)$, $m_p$ corresponds to the renormalizing term used in Extreme theory. For instance, for the  the Gamma measure, Bobkov's Theorem yields

\begin{prop}
Let $X_1,\ldots,X_n$ be i.i.d Gamma random variables. Set $M_n=\max_{i=1,\ldots,n}X_i$, then for every  $t\geq 0$ and every $n\geq 1$
 
$$
\p(M_n-\log n\geq t)\leq Ce^{-ct}
$$

and

$$
\p(M_n-\log n\leq -t)\leq Ce^{-e^{ct}}
$$
with $C,c>0$ are numerical constants.
\end{prop}

These non-asymptotic deviations inequalities express the correct tail behaviour of the maximum of Gamma random variables  (which belongs to the Gumbel's domain of attraction). Furthermore, such inequalities imply that $\p(|M_n-\log n|\geq t)\leq Ce^{-ct}$, which can be integrated to recover the fact (that can be easily obtained from Poincar\'e's inequality) that ${\rm Var}(M_n)\leq C$.\\
\newline
All of this should be obtained for the maximum of absolutes values of independent and identically distributed standard Gaussian random variables. The details are left to the reader. Recall that such kind of inequality as already been obtained by Schetchtman in \cite{Sch}.

\subsection{Comparison with existing literature}
In this section we compare our main results with recent articles which produce Superconcentration for i.i.d. random variables by other means.\\

\subsubsection{Renyi's representation and order statistics}
The authors of \cite{BT} combined three different arguments to bound the variance (or to obtain deviation inequalities)  of order statistics from a sample of i.i.d. random variables. More precisely, let $X_1,\ldots X_n$ be real i.i.d. random variables. Denote the associated order statistics by 

$$
X_{(1)}>\ldots> X_{(n)}.
$$

\noindent In their article \cite{BT}, the authors obtained the following result 

$$
{\rm Var}(X_{(k)})\leq \frac{2}{k}\E\bigg[\frac{1}{\kappa_\mu(X_{(k+1)})^2}\bigg], \quad k=1,\ldots,n.
$$

\noindent Their scheme of proof is based on Renyi's representation formula (cf. \cite{DeHaan}), which allow one to express order statistics in terms of renormalized sums of i.i.d Exponential random variables. They combined this representation with Efron-Stein's inequality (cf. \cite{BLM}) and Harris's negative association (to do so they must assume that the function $\kappa_\mu$ is non-increasing) in order to bound from above the variance of $X_{(k)},\,k=1,\ldots,n$.\\
\newline
They also obtained right deviation inequalities (around the mean) in a Gaussian setting. That is to say, if $X_i=|Y_i|$  with $\mathcal{L}(Y_i)=\mathcal{N}(0,1)$ for every $i=1,\ldots n$ and $U(s)=\Phi^{-1}(1-1/(2s))$, with $\Phi$ the distribution function of a standard Gaussian random variables, they obtained

$$
\p\bigg(X_{(1)}-\E[X_{(1)}]\leq t/(3U(n)+\sqrt{t}/U(n)+\delta_n\bigg)\leq e^{-t},\,t\geq 0
$$

\noindent with $\delta_n>0$ and $[U(n)]^3\delta_n\to\frac{\pi^2}{12}$ as $n\to \infty$.\\
\newline
The major drawback of this approach is that it can only be performed on order statistics. Our method seems to be more fexible and allows one to recover (from the measure $\nu$) Poincar\'e's inequality (for the measure of interest $\mu$) when the transport map is Lipschitz. It is also clear that the hypothesis (non-increasing) on the function $\kappa_\mu$ is not necessary to obtain upper bound on the variance. We have shown that this argument can only be used to reach exponential deviation inequalities. On this matter, Berstein's type of deviation inequality from \cite{BT} is more precise than ours, but it does not give back a relevant bound on the variance after integration. It is also surprising that the authors \cite{BT} did not deal with the more classical standard Gaussian case (without absolute value).

\subsubsection{Hypercontractive approach and semigroup interpolations}

The comparison with the hypercontractive approach is straightforward.  On one hand the hypercontractive approach can be used to deal with correlated Gaussians vectors (cf. \cite{Chatt1, KT, KT2}). On the other hand, the hypercontractive method can not reach any decay faster than $1/\log n$ and  can only provide an exponential decay at the level of concentration inequalities. For instance, it does not seem possible to show, with hypercontractive arguments, that neither the variance of the Median of a standard Gaussian sample is of order  $1/n$ nor to obtain the right order of the fluctuations of log-concave measure with potential  $V(x)=|x|^\alpha$ when $\alpha>2$ (notice also that hypercontractivity is not satisfied when $0<\alpha<1$).

\subsubsection{Comparison with Talagrand's inequality}
This section's purpose is to compare Proposition \ref{prop.poincare.poids.gaussien} with the following result. 

\begin{prop}[Talagrand]
Let  $f\,:\,\R^n\to \R$ be smooth enough, then it holds 

\begin{equation}\label{Tal2}
{\rm Var}_{\gamma_n}(f)\leq C\sum_{i=1}^n\frac{\|\partial_i f\|_2^2}{1+\log\bigg( \frac{\|\partial_if\|_2}{\|\partial_if\|_1}\bigg)},
\end{equation}
\end{prop}
\begin{rem}
This inequality was originally proved in \cite{TalL1L2} and have been a major tool in Superconcentration theory (cf. \cite{Chatt1, KT, KT2}).
\end{rem}

 To this task, it is enough to deal with the dimension one case.  Such inequalities are not comparable as it can be seen on the following functions $f_M$ and $f_{\epsilon}$. Indeed, let $M>0$ be and define the function $f_M$ by
 
 $$
 f_M(x)=\bigg(\int_0^xe^{t^2/4}1_{[-M,M]}(t)dt\bigg)/\|f'_M\|_1,\quad x\in\R.
 $$
 
 \noindent and, for every $0<\epsilon<1$, consider the function $f$, defined by 

 $$
f_\epsilon(x)=
\left\{
\begin{array}{ll}
\frac{|x|}{\epsilon}+1,\quad |x|\leq \epsilon\\
0,\quad |x|>\epsilon,\\
\end{array}
\right.
$$

\noindent Then, it is enough to choose $\epsilon=1/2n, n\geq1$. \\

\textit{Aknowledgment. This work has been done during my Ph.D and I would like to thank my Ph.D advisor M. Ledoux for fruitful discussions. Also, I would like to thank N. Gozlan for several comments and precious remarks. }

\end{document}